\documentclass{amsart}

\usepackage{epsfig}
\usepackage{psfrag}

\usepackage{infix-RPN}
\usepackage{pst-plot,pst-infixplot,pstricks,graphicx}
\usepackage{amssymb,latexsym,amsthm,amsfonts,color,fancyhdr}

\newtheorem{theorem}{Theorem}[section]

\newtheorem{lemma}[theorem]{Lemma}

\newtheorem{conjecture}[theorem]{Conjecture}

\theoremstyle{remark}
\newtheorem{remark}[theorem]{Remark}

\numberwithin{equation}{section}

\begin{document}

\title[A characterization theorem for orthogonal polynomials]
{A characterization of continuous $q$-Jacobi, Chebyshev of the first kind and Al-Salam Chihara polynomials}

\author{K. Castillo, D. Mbouna, and J. Petronilho}
\address{CMUC, Department of Mathematics, University of Coimbra, 3001-501 Coimbra, Portugal}
\email[K. Castillo]{kenier@mat.uc.pt}
\email[D. Mbouna]{dmbouna@mat.uc.pt}

\date{\today}

\thanks{Our beloved friend and mentor J. Petronilho passed away unexpectedly on August 27, 2021 while we were preparing the final version of this manuscript. To him, our eternal gratitude.}


\date{\today}


\maketitle


\begin{abstract}
The purpose of this note is to characterize those orthogonal polynomials sequences $(P_n)_{n\geq0}$ for which 
$$
\pi(x)\mathcal{D}_q P_n(x)=(a_n x+b_n)P_n(x)+c_n P_{n-1}(x),\quad n=0,1,2,\dots,
$$
where $\mathcal{D}_q$ is the Askey-Wilson operator, $\pi$ is a polynomial of degree at most 2,
and $(a_n)_{n\geq0}$, $(b_n)_{n\geq0}$ and $(c_n)_{n\geq0}$ are sequences of
complex numbers such that $c_n\neq0$ for $n=1,2,\dots$.
\end{abstract}

\section{Introduction and main results}
Let $\pi$ be a nonzero polynomial of degree at most 2 and
consider three sequences of complex numbers $(a_n)_{n\geq0}$, $(b_n)_{n\geq0}$ and $(c_n)_{n\geq0}$.
Al-Salam and Chihara \cite{Al-SalamChihara1972} proved that the only
orthogonal polynomial sequences (OPS), say $(P_n)_{n\geq0}$,
that satisfy 
\begin{align}\label{Al-Salam-Chihara}
\pi (x)\,DP_n (x)= (a_n x+b_n)P_n (x) + c_n P_{n-1} (x) , 
\end{align}
are those of Hermite, Laguerre, Jacobi, and Bessel.
Here $D$ denotes the standard derivative with respect to $x$.
Consider now (\ref{Al-Salam-Chihara}) with $D$ replaced by the Askey-Wilson operator,
\begin{align}
(\mathcal{D}_q  f)(x)=\frac{\breve{f}\big(q^{1/2} z\big)
-\breve{f}\big(q^{-1/2} z\big)}{\breve{e}\big(q^{1/2}z\big)-\breve{e}\big(q^{-1/2} z\big)},\quad
z=e^{i\theta}, \label{0.3}
\end{align}
where $\breve{f}(z)=f\big((z+1/z)/2\big)=f(\cos \theta)$ for each polynomial $f$ and $e(x)=x$.
Here $0<q<1$ and $\theta$ is not necessarily a real number (see \cite[p.\,300]{IsmailBook2005}).
The following conjecture is the first part of \cite[Conjecture 24.7.8]{IsmailBook2005}, rewritten using the \eqref{TTRR} below.

\begin{conjecture}\label{IsmailConj2478}
Let $(P_n)_{n\geq0}$ be a monic OPS and $\pi$ be a polynomial of degree at most $2$
which does not depend on $n$. If $(P_n)_{n\geq0}$ satisfies
\begin{align}
\pi(x)\mathcal{D}_q P_n (x)= (a_n x+b_n)P_n (x) + c_n P_{n-1} (x), \label{0.2Dq}
\end{align}
then $(P_n)_{n\geq0}$ are continuous $q$-Jacobi polynomials, Al-Salam-Chihara polynomials,
or special or limiting cases of them.
\end{conjecture}

Ismail himself proved that the continuous $q$-Jacobi polynomials 
indeed satisfy (\ref{0.2Dq}) for suitable polynomial $\pi$ and parameters $a_n$, $b_n$, and $c_n$
(cf. \cite[Theorem 15.5.2]{IsmailBook2005}). Al-Salam \cite{Al-Salam1995} proved Conjecture \ref{IsmailConj2478} for $\pi(x)=1$ (see \cite[Proposition 5.3.1]{DM-Thesis} for an alternative proof of this result based on the approach developed in this note),
by characterizing the Rogers $q$-Hermite polynomials, $P_n(x)=H_n(x|q)$, as the only OPS that fulfill $\mathcal{D}_q P_n=c_nP_{n-1}$ for $n=1,2,\ldots$. Recall that the monic continuous $q$-Jacobi polynomials, $\widehat{P}_n^{(a,b)}(x|q)$,
depend on two real parameters $a$ and $b$, and they
are characterized by the three-term recurrence relation 
$$
\begin{array}{rcl}
x\widehat{P}_n^{(a,b)}(x|q)&=&\widehat{P}_{n+1}^{(a,b)}(x|q)
+\mbox{$\frac12$}\,\big(q^{(2a+1)/4}+q^{-(2a-1)/4}-y_n-z_n\big)\,\widehat{P}_n^{(a,b)}(x|q) \\ [0.25em]
&&\displaystyle\, +\,\mbox{$\frac14$}y_{n-1}z_n\,\,\widehat{P}_{n-1}^{(a,b)}(x|q),
\end{array}
$$
$n=0,1,\ldots$, being
\begin{align*}
& y_n=\frac{(1-q^{n+a+1})(1-q^{n+a+b+1})(1+q^{n+(a+b+1)/2})(1+q^{n+(a+b+2)/2})}
{q^{(2a+1)/4}(1-q^{2n+a+b+1})(1-q^{2n+a+b+2})} , \\
& z_n=\frac{q^{(2a+1)/4}(1-q^n)(1-q^{n+b})(1+q^{n+(a+b)/2})(1+q^{n+(a+b+1)/2})}
{(1-q^{2n+a+b})(1-q^{2n+a+b+1})},
\end{align*}
while the monic Al-Salam-Chihara polynomials, $Q_n(x;c,d|q)$,
which also depend on two parameters $c$ and $d$, are characterized by
$$
\begin{array}{rcl}
xQ_{n}(x;c,d|q)&=&Q_{n+1}(x;c,d|q)+\,\mbox{$\frac12$}\,(c+d)q^n\,Q_{n}(x;c,d|q) \\ [0.25em]
&&\displaystyle\, +\,\mbox{$\frac14$}\,(1-cd q^{n-1})(1-q^n)\,Q_{n-1}(x;c,d;q),
\end{array}
$$
$n=0,1,\ldots$, provided we define $\widehat{P}_{-1}^{(a,b)}(x|q)=Q_{-1}(x;c,d|q)=0$
(see e.g. \cite{IsmailBook2005}).
Further, up to normalization, the Rogers $q$-Hermite polynomials are the special case $c=d=0$ of the Al-Salam-Chihara polynomials.

The following two theorems summarize the main results of this note and, together with Al-Salam's theorem, give positive answer to Conjecture \ref{IsmailConj2478}. 

\begin{theorem}\label{T1}
The Al-Salam Chihara polynomials with nonzero parameters $c$ and $d$ such that $c/d=q^{\pm 1/2}$ are the only
OPS  satisfying \eqref{0.2Dq} for $\deg \pi=1$. 
\end{theorem}

\begin{theorem}\label{T2}
The Chebyschev polynomials of the first kind and the continuous $q$-Jacobi polynomials are the only OPS satisfying \eqref{0.2Dq} for $\deg \pi=2$. 
\end{theorem}

The two previous theorems give more than a simple answer to Conjecture \ref{IsmailConj2478}. In fact, we now know, for instance, that Al-Salam Chihara polynomials appear only when $\deg\pi = 1 $ and for $c/d=q^ {\pm 1/2} $, or that continuous $ q $-Jacobi polynomials are exclusively related to the case $ \deg\pi = 2 $, or that Chebyschev polynomials of the first type are the only ``limiting case'' that satisfy Conjecture \ref{IsmailConj2478}. 

\begin{remark} Koornwinder \cite[Section 4]{K07} gave a structure formula (and resulting
lowering formula) for Askey-Wilson polynomials. Necessarily, because of
Conjecture \ref{IsmailConj2478} settled in the present paper, the left-hand side in
Koornwinder's formula has a more general form than in \eqref{0.2Dq}.
With the usual parameters $a$, $b$, $c$, $d$ for Askey-Wilson polynomials,
substitution of $c=q^{1/2}$, $d=-q^{1/2}$ in these formulas gives similar formulas
for continuous q-Jacobi polynomials with $q$ replaced by $q^2$
(see \cite[Section 5]{K07}). In a personal communication Koornwinder pointed out
to us that Ismail's and his formulas in the continuous q-Jacobi case
are the same modulo polynomial multiples of the second order q-difference
formula and the three-term recurrence relation for continuous
q-Jacobi polynomials.\end{remark}

Since this work is an application of the theory developed in our previous paper \cite{KCDMJPArxiv2021a}, in the next sections, we  suppose that the reader has  \cite{KCDMJPArxiv2021a} at hand and we shall use its notation, definitions, and results.

\section{Preliminary results}
Taking $e^{i\theta}=q^s$ in \eqref{0.3}, $\mathcal{D}_q$ reads
\begin{equation*}
\mathcal{D}_q f(x(s))= \frac{f\big(x(s+\frac{1}{2})\big)-f\big(x(s-\frac{1}{2})\big)}{x(s+\frac{1}{2})-x(s-\frac{1}{2})},
\quad x(s)=\mbox{$\frac12$}(q^s +q^{-s}).
\end{equation*}
We define an operator $\mathcal{S}_q:\mathbb{R}[x]\to\mathbb{R}[x]$ by
\begin{equation*}
\mathcal{S}_q f(x(s))=\frac{f\big(x(s+\frac{1}{2})\big)+f\big(x(s-\frac{1}{2})\big)}{2}.
\end{equation*}
Hereafter, we denote $X=x(s)=(q^{s}+q^{-s})/2$ with $0<q<1$. Recall that a monic OPS $(P_n)_{n\geq 0}$ satisfies the following three term recurrence relation (TTRR):
\begin{align}
XP_n(X)=P_{n+1}(X)+B_nP_n(X)+C_nP_{n-1}(X), \quad n=0,1,2,\dots,\label{TTRR}
\end{align}
with $P_{-1}(X)=0$ and $B_n\in \mathbb{C}$ and $C_{n+1} \in \mathbb{C}\setminus \left\lbrace 0\right\rbrace$ for each $n=0,1,2,\ldots$. We start by showing that all monic OPS, $(P_n)_{n\geq 0}$, satisfying \eqref{0.2Dq} belongs to a well known class of OPS and then we prove that the coefficients of the associated TTRR satisfy a system of non linear equations that will be solved in the next section.

\begin{lemma}\label{conject-is-dx-classical}
Let ${\bf u} \in \mathcal{P}^* $ be a regular functional such that its corresponding monic OPS $(P_n)_{n\geq0}$ satisfies \eqref{0.2Dq} subject to the condition $c_n\neq 0$ for $n=1,2,\ldots$. Then ${\bf u}$ is $x$-classical; that is $\mathcal{D}_q(\phi {\bf u})=\mathcal{S}_q(\psi {\bf u})$ with $\deg \phi \leq 2$ and $\deg \psi \leq 1$. Moreover, $\psi$ and $\phi$ are polynomials given by
\begin{align}\label{value-d-e-a}
\psi(X)=X-B_0,\quad \phi(X)=(\mathfrak{a}X-\mathfrak{b})(X-B_0 )-(\mathfrak{a}+\alpha)C_1\; ,
\end{align}
where
\begin{align}\label{coef-a-b-conject-classical}
\mathfrak{a}= \frac{(a_2 C_2 +c_2)C_1}{(a_1 C_1 +c_1)C_2} -\alpha,
\quad \mathfrak{b}= -B_0 +(\mathfrak{a}+\alpha)B_1 -\frac{b_1 +a_1 B_1}{c_1 +a_1 C_1}C_1.
\end{align}
(Here $B_0$, $B_1$, $C_1$, and $C_{2}$ are coefficients of the TTRR \eqref{TTRR} satisfied by $(P_n)_{n\geq0}$.)
\end{lemma}

\begin{proof}
Let $({\bf a}_n)_{n\geq0}$ be the dual basis associated to the monic OPS $(P_n)_{n\geq0}$. We claim that 
\begin{align}\label{def-R1}
\mathcal{D}_q (\pi {\bf u}) = R_1 {\bf u},\quad R_1(X) = -\frac{a_1 C_1 +c_1}{C_1}(X-B_0),
\end{align}
with $a_1C_1+c_1\neq 0$. Indeed, using \eqref{0.2Dq} and \eqref{TTRR}, we have
\begin{align*}
\left\langle \mathcal{D}_q (\pi {\bf a}_0),P_j \right\rangle &=-\left\langle {\bf a}_0,\pi \mathcal{D}_q P_j \right\rangle 
=-a_j \delta_{0,j+1}-(a_j B_j +b_j)\delta_{0,j}-(c_j +a_j C_j)\delta_{1,j},
\end{align*}
for fixed $j \in \mathbb{N}_0$. Taking $n=0$ in \eqref{0.2Dq}, we find $a_0 =b_0=0$, and since 
$\left\langle {\bf u},P_n ^2 \right\rangle {\bf a}_n =P_n {\bf u}$ and 
$C_{n+1}=\left\langle {\bf u},P_{n+1} ^2 \right\rangle /\left\langle {\bf u},P_n ^2 \right\rangle$, we obtain
$$
\mathcal{D}_q (\pi {\bf a}_0) =\sum_{j=0} ^{\infty} \left\langle \mathcal{D}_q (\pi {\bf a}_0),P_j \right\rangle {\bf a}_j =-(c_1 +a_1 C_1){\bf a}_1.
$$ 
If $c_1 +a_1 C_1 =0$, then ${\bf D}_x (\pi {\bf u})=0$, hence $0=\left\langle \mathcal{D}_q (\pi {\bf u}),f \right\rangle=-\left\langle \pi{\bf u},  \mathcal{D}_q f \right\rangle$, $\forall f\in \mathcal{P}$. This implies $\pi {\bf u}={\bf 0}$. But this is impossible, since $\pi \not= 0$ and ${\bf u}$ is regular. So $c_1 +a_1 C_1 \neq 0$. Hence \eqref{def-R1} holds. Applying $\mathcal{D}_q$ to both sides of (\ref{TTRR}), and using \cite[(2.28)]{KCDMJPArxiv2021a}, yields
\begin{align*}
\mathcal{S}_q P_n(X) =-\alpha X\mathcal{D}_q P_n(X) +
\mathcal{D}_q P_{n+1}(X) +B_n \mathcal{D}_q P_n(X) +C_n \mathcal{D}_q P_{n-1}(X).
\end{align*}
Multiplying both sides of this equality by $\pi(X)$ and using (\ref{0.2Dq}) and (\ref{TTRR}), we obtain
\small{
\begin{align}\label{piSx-eq}
\pi(X)\mathcal{S}_q P_n (X)=r_n ^{[1]} P_{n+2}(X)+r_n ^{[2]} P_{n+1}(X)+r_n ^{[3]}P_n (X)+r_n ^{[4]}P_{n-1}(X)+r_n ^{[5]}P_{n-2}(X) 
\end{align}
}
for each $n=0,1,2,\ldots$, where
\begin{align*}
r_n ^{[1]}&= a_{n+1}-\alpha a_n,\\
 r_n ^{[2]}&= g_{n+1}-\alpha g_n+a_n(B_n-\alpha B_{n+1}),\\
r_n ^{[3]}&= s_{n+1}-\alpha s_n+g_n(1-\alpha) B_{n}+a_{n-1}C_n-\alpha a_nC_{n+1},\\
r_n ^{[4]}&= (g_{n-1}-\alpha g_n)C_{n}+s_n(B_n-\alpha B_{n-1}),\\
 r_n ^{[5]}&= C_ns_{n-1}-\alpha C_{n-1}s_n,
\end{align*}
and $g_n =b_n +a_n B_n$, $s_n =c_n +a_n C_n$.
For a fixed $j\in \mathbb{N}_0$, using \eqref{piSx-eq} we obtain 
$$
\left\langle \mathcal{S}_q (\pi {\bf a}_0),P_j  \right\rangle = \left\langle {\bf a}_0, \pi \mathcal{S}_q P_j\right\rangle 
= r_j ^{[1]}\delta_{0,j+2} +r_j ^{[2]}\delta_{0,j+1} +r_j ^{[3]}\delta_{0,j}+r_j ^{[4]}\delta_{0,j-1}+r_j ^{[5]}\delta_{0,j-2}.
$$
Therefore, 
$$  \mathcal{S}_q (\pi {\bf a}_0) =\sum_{j=0} ^{\infty} \left\langle \mathcal{S}_q (\pi {\bf a}_0),P_j  \right\rangle {\bf a}_j 
= r_0 ^{[3]} {\bf a}_0 +r_1 ^{[4]} {\bf a}_1 +r_2 ^{[5]} {\bf a}_2, $$ 
and so
\begin{align}\label{def-R2}
\mathcal{S}_q (\pi {\bf u}) = R_2 {\bf u},\quad R_2(X) = r_0 ^{[3]} +\frac{r_1 ^{[4]}}{C_1} P_1 (X)+\frac{r_2 ^{[5]}}{C_2 C_1}P_2 (X).
\end{align}
Next, on the first hand, applying successively \eqref{def-R2}, \cite[(2.47)]{KCDMJPArxiv2021a} and \eqref{def-R1}, we obtain
\begin{align}\label{first-hand}
\mathcal{D}_q (R_2 {\bf u}) = \frac{2\alpha ^2 -1}{\alpha} \mathcal{D}_q (R_1 {\bf u}) +\frac{\texttt{U}_1}{\alpha} \mathcal{D}_q  (R_1 {\bf u}).
\end{align}
On the other hand, using \cite[(2.34)]{KCDMJPArxiv2021a} with $f=\texttt{U}_1$, we obtain 
$$\alpha \mathcal{D}_q (\texttt{U}_1 R_1 {\bf u})
=\texttt{U}_1 \mathcal{D}_q (R_1 {\bf u}) +(\alpha^2 -1)\mathcal{S}_q (R_1 {\bf u}).$$
Thus, from \eqref{first-hand} we obtain
$
\mathcal{D}_q \Big((R_2 -\texttt{U}_1 R_1){\bf u}\Big)=\mathcal{S}_q (\alpha R_1 {\bf u}).
$
This leads us to define
$$
\psi(X)=X-B_0 ,\quad \phi(X)=-\frac{C_1}{\alpha (c_1 +a_1C_1)}\Big( R_2 (X)-\texttt{U}_1 (X)R_1 (X)\Big).
$$
Clearly, $\deg \psi =1$, $\deg \phi \leq 2$ and  $\mathcal{D}_q(\phi {\bf u})=\mathcal{S}_q (\psi {\bf u})$.
Finally, since $a_0 =0=b_0$, and setting (without lost of generality) $c_0=0$ and $C_0=0$, we have
\begin{align*}
&\phi''(0)
=\frac{2(c_2 +a_2C_2)C_1}{(c_1 +a_1C_1)C_2}-2\alpha,\;\phi'(0)= (1-\mathfrak{a})B_0  -(\mathfrak{a}+\alpha)B_1 +\frac{b_1 +a_1 B_1}{c_1 +a_1 C_1}C_1,  \\
&\phi(0)=-(\mathfrak{a}+\alpha)C_1 -B_0 \left(  B_0 -(\mathfrak{a}+\alpha)B_1 +\frac{b_1 +a_1 B_1}{c_1 +a_1 C_1}C_1 \right),
\end{align*} 
and the proof is complete.
\end{proof}

\begin{lemma}\label{system-in-genral-form-}
Let $(P_n)_{n\geq0}$ be a monic OPS satisfying \eqref{0.2Dq}. Then the coefficients $B_n$ and $C_{n}$ of the TTRR \eqref{TTRR} satisfied by $(P_n)_{n\geq0}$ fulfill the following system of difference equations: 
\begin{align}
&\label{eq1S-eq2S} a_{n+2}-2\alpha a_{n+1}+a_n=0,\; t_{n+2} -2\alpha t_{n+1}+t_{n} =0,\\
&\nonumber \quad\quad\quad\quad\quad t_n =\frac{c_n}{C_n} =k_1 q^{n/2}+k_2q^{-n/2},  \\
&r_{n+3}B_{n+2} -(r_{n+2} +r_{n+1})B_{n+1}  +r_n B_n =0,\quad r_n= t_n +a_n -a_{n-1},  \label{eq3S} \\
&\label{eq4S} r_n \left( B_{n} ^2 -2\alpha B_{n}B_{n-1} +B_{n-1}^2   \right)\\
&\nonumber \quad \quad\quad\quad\quad=(r_{n+1}+r_{n+2})(C_{n+1}-1/4) -2(1+\alpha)r_n(C_n -1/4)\\
 &\nonumber \quad\quad\quad\quad\quad+(r_{n-1}+r_{n-2})(C_{n-1} -1/4) \\
&\label{eq5S} (1 -\alpha^2)b_n=2(1-\alpha)(a_nB_n+b_n)B_n ^2+(t_{n+1}\\
&\nonumber  \quad\quad\quad\quad\quad+a_{n+1}-a_{n+2})B_{n+1}C_{n+1}+(t_n+a_{n-1}-a_{n-2})B_{n-1}C_n \\
&\nonumber  \quad\quad\quad\quad\quad+\Big[(2a_n-a_{n+2}-a_{n-1})C_{n+1}+(2a_n -a_{n+1}-a_{n-2})C_n\\
&\nonumber \quad\quad\quad\quad\quad +(1-2\alpha)(c_n+c_{n+1})+(\alpha ^2 -1)a_n  \Big]B_n +2(b_n -\alpha b_{n+1} )C_{n+1}\\
&\nonumber \quad \quad\quad\quad\quad +2(b_n -\alpha b_{n-1})C_{n}.
\end{align}
In addition, the following relations hold:
\begin{align}
&b_n =\gamma_n,\quad c_n =\left(b_n -b_{n-1}  \right)\sum_{j=0} ^{n-1} B_j  +\pi(0)b_n,\quad\quad~ ~~\mbox{\rm if}~~ deg\,\pi=1, \label{power-pi1}\\
&a_n=\gamma_n,\quad b_n =\left(a_n -a_{n-1}  \right)\sum_{j=0} ^{n-1} B_j +\pi'(0)a_n,\quad ~~\mbox{\rm if}~~deg\,\pi=2. \label{power-pi2}
\end{align}
\end{lemma}

\begin{proof}
Applying the operator $\mathcal{S}_q$ to both sides of (\ref{TTRR}) and using \cite[(2.29)]{KCDMJPArxiv2021a}, we deduce
$\texttt{U}_2(X) \mathcal{D}_q P_n(X) + \alpha X\mathcal{S}_q P_n(X) =\mathcal{S}_q P_{n+1}(X) +B_n \mathcal{S}_q P_n(X) +C_n \mathcal{S}_q P_{n-1}(X).$
Multiplying both sides of this equality by $\pi(X)$ and then using successively \cite[(2.25)]{KCDMJPArxiv2021a} (for $\mathfrak{c}_3=0$ and $\mathfrak{c}_1=\mathfrak{c}_2=1/2$), (\ref{0.2Dq}), (\ref{piSx-eq}), and (\ref{TTRR}), we obtain
a vanishing linear combination of the polynomials $P_{n+3}, P_{n+2}, \dots,P_{n-3}$.
Thus, setting
$t_n=c_n/C_n$ for $n=1,2,3,\ldots$,  
after straightforward computations we obtain \eqref{eq1S-eq2S} together with the following equations:
\begin{align}
&\label{seq3S} (a_{n+1}-a_{n+2})B_{n+1} +(a_n -a_{n-1})B_n +b_{n+2}-2\alpha b_{n+1}+b_n =0 , \\
&\label{seq4S} (a_{n+1}-a_{n+2}-t_{n+2})B_{n+1} +(a_{n}-a_{n-1}+t_{n+1}+t_{n})B_n -t_{n-1} B_{n-1} \\ 
&\nonumber  \quad\quad  +b_{n+1}-2\alpha b_{n}+b_{n-1} =0 , \\
&\label{seq5S} (a_{n+1}-a_{n+2})B_{n+1}^2+2(1-\alpha)a_nB_n^2+(a_n-a_{n-1})B_n B_{n+1}+ (a_n -a_{n+2})C_{n+1}  \\ 
&\nonumber \quad +(b_{n+1}+b_n-2\alpha b_{n+1})B_{n+1} +(b_{n+1}+b_n -2\alpha b_{n} )B_n  +(a_n - a_{n-2})C_n\\ 
&\nonumber \quad +c_{n+2}-2\alpha c_{n+1}+c_n  = (1-\alpha ^2)a_n ,\\
&\label{seq6S} (2(1-\alpha)a_n+t_n)B_{n}^2 + (t_n +a_{n-1}-a_{n-2})B_{n-1}^2+(b_{n} +b_{n-1} -2\alpha b_{n})B_{n} \\ 
&\nonumber \quad +(a_n-t_{n-1}-t_{n+1}-a_{n+1})B_n B_{n-1} +(b_{n-1}+b_n -2\alpha b_{n-1} )B_{n-1} \\ 
&\nonumber \quad +(a_n-a_{n+2}-t_{n+2}-t_{n+1})C_{n+1}+( 2(1+\alpha)t_n+a_n-a_{n-2})C_n  \\
&\nonumber \quad -(t_{n-2}+t_{n-1})C_{n-1} +c_{n+1}-2\alpha c_{n}+c_{n-1} = (1-\alpha^2)(t_n+a_n)  ,\\
&\label{seq7S} 2(1-\alpha)a_nB_n^3 +2(1-\alpha)b_nB_{n}^2 +\left[ (2a_n -a_{n+2}-a_{n-1})C_{n+1}\right.\\
&\nonumber \quad \left. + (2a_n -a_{n+1}-a_{n-2})C_n  + c_{n+1}-2\alpha c_n +c_n -2\alpha c_{n+1}  - (1-\alpha^2) a_n \right] B_{n} \\ 
&\nonumber \quad +(c_{n+1}+a_{n+1}C_{n+1}-a_{n+2}C_{n+1})B_{n+1} +(c_n +a_{n-1}C_{n}-a_{n-2}C_n)B_{n-1} \\ 
&\nonumber \quad +2(b_n -\alpha b_{n+1})C_{n+1} +2(b_n -\alpha b_{n-1} )C_n   = (1 -\alpha^2)b_n .
\end{align}

\eqref{eq3S} (respectively, \eqref{eq4S}) is obtained by shifting $n$ to $n+1$ in \eqref{seq4S} (respectively, \eqref{seq6S}) and combining it with \eqref{seq3S} (respectively, \eqref{seq5S}) and by using \eqref{eq1S-eq2S}. \eqref{eq5S} follows from \eqref{eq1S-eq2S} and \eqref{seq7S}. Now suppose that $deg\pi =2$. Using \eqref{TTRR}, we may write \small{$$ P_n (X)=X^n -X^{n-1}\sum_{j=0} ^{n-1} B_j +w_nX^{n-2}+\cdots,$$} for some complex sequence $(w_n)_{n\geq0}$. Using \cite[(2.40)]{KCDMJPArxiv2021a} (for $\mathfrak{c}_3=0$ and $\mathfrak{c}_1=\mathfrak{c}_2=1/2$), we compare the two first coefficients of higher power of $n$ in both side of \eqref{0.2Dq} to deduce \eqref{power-pi2}. \eqref{power-pi1} is obtained in a similar way and this completes the proof.  
\end{proof}

For the next results of this section, we distinguish two cases according as $\deg\, \pi=1$ or $\deg\, \pi=2$.

\subsection{Case $\deg\, \pi=1$}
In this case, \eqref{0.2Dq} can be rewritten as
\begin{align}
(X-r)\mathcal{D}_q P_n(X)=b_nP_n(X)+c_nP_{n-1}(X),\quad n=0,1,2,\ldots, \label{case-pi-1}
\end{align} 
where $r\in \mathbb{C}$. 
\begin{lemma}\label{uniqueness-case-1}
Let $(P_n)_{n\geq0}$ be a monic OPS satisfying \eqref{case-pi-1}. Then 
\begin{align}
\left( c_2 C_1-q^{-1/2}c_1C_2\right) \left(c_2C_1-q^{1/2}c_1C_2\right)  = 0. \label{uniqueness-cond-case-1}
\end{align}
\end{lemma}
\begin{proof}
Since $(P_n)_{n\geq0}$ satisfies \eqref{case-pi-1}, then $a_n =0$ for each $n=0,1,2,\ldots$ and by \eqref{eq1S-eq2S}  and \eqref{power-pi1}, we have 
\begin{align}
t_n =k_1q^{n/2}+k_2q^{-n/2},~ k_1 = \frac{c_2 C_1-q^{-1/2}c_1C_2}{(q-1)C_1C_2},~ k_2= \frac{c_2C_1-q^{1/2}c_1C_2}{(q^{-1}-1)C_1C_2} .\label{tn-pi-1}
\end{align}
Suppose, contrary to our claim, that \eqref{uniqueness-cond-case-1} does not hold. This means that $k_1k_2\neq 0$. Taking successively $n=1$ and $n=2$ in \eqref{case-pi-1}, and using \eqref{TTRR}, \cite[(2.40)]{KCDMJPArxiv2021a} (for $\mathfrak{c}_3=0$ and $\mathfrak{c}_1=\mathfrak{c}_2=1/2$), we have $b_1=1$, $b_2=2\alpha$, $r=B_0-c_1$, $c_2=(2\alpha -1)(B_1+B_0)-2\alpha r$ and $r(B_1+B_0)=-c_2B_0 +2\alpha (B_0B_1-C_1)$.  
Hence
\begin{align}\label{B1-case-1}
c_1=B_0 -r,\quad B_1=(2\alpha-1)B_0 +2\alpha \frac{C_1}{c_1}.
\end{align}
We claim that 
\begin{align}
B_n = \frac{t_0t_1B_0}{t_nt_{n+1}},\quad n=0,1,2,\dots, \label{bn-reduced}
\end{align}
with  $B_0\neq 0$. Indeed, writing \eqref{eq3S} as $t_{n+3}B_{n+2}-t_{n+1}B_{n+1}=t_{n+2}B_{n+1}-t_{n}B_{n}$ and proceeding in a recurrent way, we have 
\begin{align}\label{explicit-Bn}
B_n =\frac{t_0t_1B_0q^{n/2} +K_b\Big(k_1q^n +(k_2-k_1q^{-1/2})q^{n/2}-k_2q^{-1/2}\Big)}{(k_1q^{n} +k_2)(k_1q^{n+1} +k_2)}q^{(n+1)/2},
\end{align}
where $K_b=(t_2B_1 -t_0B_0)/(1-q^{-1/2})$.
Since $k_2\neq 0$ and $0<q<1$, we obtain $\lim_{n \rightarrow \infty} q^{-n/2}B_n=-K_b/k_2$ and consequently we have $K_b=0$ by applying limit of the same expression using \cite[(4.4)]{KCDMJPArxiv2021a} ( for $\mathfrak{c}_3=0$ and $\mathfrak{c}_1=\mathfrak{c}_2=1/2$). This implies that $B_1=t_0B_0/t_2$. 
If $B_0=0$, then we find $B_1=0$ which is in contradiction with the second equation in \eqref{B1-case-1}. Then \eqref{bn-reduced} is proved. Note that, from \eqref{coef-a-b-conject-classical} and using \eqref{tn-pi-1}, we obtain
\begin{align}\label{valeu-a-pi-1-new}
\mathfrak{a}=\frac{c_2C_1}{c_1C_2}-\alpha =\frac{t_2}{t_1}-\alpha=\frac{k_1q^{1/2}-k_2q^{-1/2}}{2ut_1},~u^{-1}=q^{1/2}-q^{-1/2},
\end{align}
since $a_n=0$. Using \eqref{bn-reduced}, we obtain 
$
S_n =\sum_{j=0} ^{n-1} B_j=t_1B_0 \gamma_n/t_n$ for $n=0,1,2,\ldots$.
Thus using \cite[(4.4)]{KCDMJPArxiv2021a} (for $\mathfrak{c}_3=0$ and $\mathfrak{c}_1=\mathfrak{c}_2=1/2$) to evaluate the same sum, we have 
$t_1 B_0(\mathfrak{a}\gamma_{2n-2}+ \alpha_{2n-2})=-t_n \big(\phi'(0)\gamma_{n-1} -B_0\alpha_{n-1}\big)$. This gives the following equations: 
$$(2\mathfrak{a}ut_1+k_2q^{-1/2})B_0=-2uk_1q^{1/2}\phi'(0)\; \textit{and}\;(2\mathfrak{a}ut_1-k_1q^{1/2})B_0=-2uk_2q^{-1/2}\phi'(0).$$ Taking into account that $k_1k_2\neq 0$ and using \eqref{valeu-a-pi-1-new}, we get $$|B_0+2u\phi'(0)|+|B_0-2u\phi'(0)|=0,$$
which is impossible because $B_0\neq 0$, and the lemma follows.
\end{proof}

\subsection{Case $deg\,\pi=2$}
In this case, we rewrite \eqref{0.2Dq} as
\begin{align}
(X-r)(X-s)\mathcal{D}_qP_n (X)=(a_nX+b_n)P_n (X)+c_nP_{n-1}(X),~ n=0,1,2,\dots,\label{case-pi-2}
\end{align}
where $r,s\in \mathbb{C}$ and $c_n\neq 0$ for $n=1,2,3,\ldots$.
From \eqref{power-pi2}, \eqref{eq1S-eq2S} and \eqref{eq3S}, we obtain
\begin{align}
a_n&=\gamma_n,\quad b_n=\left(\gamma_n -\gamma_{n-1}  \right)\sum_{k=0} ^{n-1}B_k -(r+s)\gamma_n,\label{bn-case-pi-2}\\
t_n&=\frac{c_n}{C_n}=k_1q^{n/2}+k_2q^{-n/2},\quad k_1 = \frac{c_2 C_1-q^{-1/2}c_1C_2}{(q-1)C_1C_2},\quad k_2= \frac{c_2C_1-q^{1/2}c_1C_2}{(q^{-1}-1)C_1C_2},\label{tn-case-pi-2}  \\
r_n&=\widehat{a}q^{n/2}+\widehat{b}q^{-n/2},\quad \widehat{a}=k_1+u(1-q^{-1/2}),\quad \widehat{b}=k_2-u(1-q^{1/2}),\label{rn-case-pi-2}
\end{align}
for $n=0,1,2,\ldots$, $u^{-1}=q^{1/2}-q^{-1/2}$. Recall that $t_0=k_1+k_2$ and so, we also define by compatibility $r_0=\widehat{a}+\widehat{b}.$

\begin{lemma}\label{case-pi-2-unicity}
Let $(P_n)_{n\geq0}$ be a monic OPS satisfying \eqref{case-pi-2}. Then 
\begin{align}
\widehat{a}~\widehat{b} (1-2\mathfrak{a}u)(1+2\mathfrak{a}u)\neq 0,  \label{cond-conject2-final} 
\end{align}
where $\widehat{a}$, $\widehat{b}$ are defined in \eqref{rn-case-pi-2} and $\mathfrak{a}$ is given in \eqref{coef-a-b-conject-classical}. 
\end{lemma}

\begin{proof}
Assume that \eqref{cond-conject2-final} does not hold. 
Suppose, for instance, that $\widehat{a}=0$. Then \eqref{rn-case-pi-2} reduces to $r_n=\widehat{b}q^{-n/2}$ for each $n=0,1,2,\ldots$. Then \eqref{eq3S} becomes
$$q^{-3/2}B_{n+2}-(q^{-1}+q^{-1/2})B_{n+1}+B_n=0,\quad n=0,1,2,\dots,$$
and so we may write
\begin{align}\label{possible-bn-pi-2}
B_n=r_0(1-q^{1/2})q^{n/2}+s_0(1-q)q^n,\quad n=0,1,2,\dots,
\end{align}
for some complex numbers $r_0$ and $s_0$.
From \eqref{coef-a-b-conject-classical}, we also have
\begin{align*}
\mathfrak{a}=\frac{(c_2+2\alpha C_2)C_1}{(c_1+C_1)C_2}-\alpha =\frac{1+r_2}{r_1}-\alpha=-\frac{1}{2u}+\frac{1}{\widehat{b}q^{-1/2}}.
\end{align*}
From \eqref{possible-bn-pi-2}, we get $S_n=\sum_{j=0} ^{n-1}B_j=r_0(1-q^{n/2})+s_0(1-q^n)$ for $n=0,1,2,\ldots$. Now we compute this sum using \cite[(4.4)]{KCDMJPArxiv2021a} (for $\mathfrak{c}_3=0$ and $\mathfrak{c}_1=\mathfrak{c}_2=1/2$) to obtain 
\begin{align}\label{equiv-bn-pi-2}
(r_0q^{n/2}+s_0q^n-r_0-s_0)d_{2n}=\gamma_{n+1}e_n, \quad n=1,2,3,\ldots,
\end{align}
where $2d_{2n}=(1+2\mathfrak{a}u)q^n +(1-2\mathfrak{a}u)q^{-n}$ and $2e_n=(-B_0+2u\phi'(0)  )q^{n/2}-(B_0+2u\phi'(0))q^{-n/2}$ for $n=0,1,2,\ldots$. 
It is easily seen that  \eqref{equiv-bn-pi-2} implies $r_0=0=s_0$ as well as $B_0=0=\phi'(0)$. Hence $B_n=0$ for $n=0,1,2,\ldots.$ In addition, using \eqref{value-d-e-a}, we obtain $\mathfrak{b}=0$. Next we apply \cite[(4.5)]{KCDMJPArxiv2021a} (for $\mathfrak{c}_3=0$ and $\mathfrak{c}_1=\mathfrak{c}_2=1/2$) to obtain
\begin{align}\label{cn-a-bar-0-pi-2}
C_{n+1}= \frac{(1-q^{n+1})(B-q^n)\Big(q^{2n+1}+( 4(q+\widehat{b})C_1-q-B) q^n +B\Big)}{4(B-q^{2n})(B-q^{2n+2})},
\end{align}
for $n=0,1,2,\ldots$, with $B=q+\widehat{b}(1-q)$, while \eqref{eq4S} reduces to
$$(q^{-1/2}+q^{-1})(C_{n+1}-1/4)-2(1+\alpha)(C_n-1/4)+(q^{1/2}+q)(C_{n-1}-1/4)=0.$$
Therefore, we may write
$C_{n+1}=\overline{r}_0q^{n/2}+\overline{s}_0q^n +1/4$ for $n=2,3,\ldots$, with $\overline{r}_0,\overline{s}_0 \in \mathbb{C}$. This is compatible with \eqref{cn-a-bar-0-pi-2} 
if and only if $C_1=1/2$, $\overline{r}_0=\overline{s}_0=0$ and $\widehat{b}=1$. This implies that 
$$
B_{n-1}=0,\quad C_1=1/2,\quad C_{n+1}=1/4,\quad n=1,2,\dots.
$$
From \eqref{b1-pi-2-case-2}-\eqref{c2-b2-pi-2-case-2} below we find $r=-s$ and so 
\begin{align}\label{c1-c2-a=0}
c_1=-r^2,\quad c_2=\alpha (1-2r^2) .
\end{align}
Since $\widehat{a}=0$ and $\widehat{b}=1$, from \eqref{rn-case-pi-2} we find $k_1=-k_2=-(1+q^{1/2})^{-1}$ and so $t_n=-(1-q^{-1/2})\gamma_n$. On the other hand we have
$$\mathfrak{a}=\frac{(c_2+2\alpha C_2)C_1}{(c_1+C_1)C_2}-\alpha =\frac{2\alpha +t_2}{1+t_1}-\alpha =\alpha .$$ 
Using \eqref{c1-c2-a=0}, we also write $\alpha =\mathfrak{a}=\alpha(5 -6r^2)/(1-2r^2)$,
hence $r^2=1$ so that $c_1=-1$ and $c_2=-\alpha$. But using what is preceding, from \eqref{tn-case-pi-2} we have $k_1=-k_2=-2u$ and therefore $t_n=-2\gamma_n$, which is in contradiction with the previous expression of $t_n$, which gives $\widehat{a}\neq 0$. The case $\widehat{b}=0$ can be treated similarly.

Assume now that $1+2\mathfrak{a}u=0$. Since $\mathfrak{a}=-\alpha+ (1+r_2)/r_1$, we obtain $\widehat{a}=-uq^{-1/2}\neq 0$.  
On the other hand, we use \cite[(4.4)]{KCDMJPArxiv2021a}  (for $\mathfrak{c}_3=0$ and $\mathfrak{c}_1=\mathfrak{c}_2=1/2$) to obtain
\begin{align*}
B_n=q^n (1+q^{-1})\Big(\mathfrak{b}u(q^n-1) +\frac{B_0}{1+q^{-1}}  \Big),
\end{align*}
for $n=0,1,2,\ldots$. This satisfies \eqref{eq3S} if and only if $\mathfrak{b}=0$ and $B_0=0$, and so $B_n=0$. Taking this into account, \cite[(4.5)]{KCDMJPArxiv2021a} (for $\mathfrak{c}_3=0$ and $\mathfrak{c}_1=\mathfrak{c}_2=1/2$) gives $$C_{n+1}=\frac{1}{4}(1-q^{n+1})\Big(1-q^n +\frac{4C_1}{1-q}q^n  \Big),\quad n=0,1,2,\dots.$$ 
This does not satisfy \eqref{eq4S} because $\widehat{a}\neq 0$ and $B_n=0$.
Hence $1+2\mathfrak{a}u\neq 0$. The case $1-2\mathfrak{a}u=0$ can be treated similarly, which proves the lemma. 
\end{proof}

\section{Proof of theorems \ref{T1} and \ref{T2}}

\begin{proof}[Proof of Theorem \ref{T1}]
Note that \eqref{uniqueness-cond-case-1} is equivalent to $k_1k_2= 0$. Suppose that $k_1=0$. By \eqref{tn-pi-1}, we have $t_n=k_2q^{-n/2}$, where $k_2=q^{1/2}c_1/C_1$. We claim that
\begin{align}\label{Bn-final-case-1}
B_n=B_0q^n =(r+c_1)q^n,\quad n=0,1,2,\ldots.
\end{align}
Indeed, \eqref{eq3S} reduces to $q^{-1/2}B_{n+2}+(1+q^{1/2})B_{n+1}+qB_n=0,~~n=0,1,2,\ldots$ and so we find $B_n =vq^n+sq^{n/2}$ for some $v,s\in \mathbb{C}$.
Moreover, since $k_1=0$, from \eqref{valeu-a-pi-1-new} we get $ \mathfrak{a}=-1/(2u)$. Hence, by \eqref{value-d-e-a}, $\phi(X)=-(\left(X+2\mathfrak{b}u \right)\left(X-B_0 \right)+2uq^{-1/2}C_1)/(2u)$ and $\psi(X)=X-B_0$. 
Therefore, using \cite[(4.4)]{KCDMJPArxiv2021a} (for $\mathfrak{c}_3=0$ and $\mathfrak{c}_1=\mathfrak{c}_2=1/2$), we obtain $B_n= q^{(2n-1)/2}(2\alpha u\mathfrak{b}(q^n -1) + q^{1/2}B_0)$. Comparing the two previous expressions for $B_n$, we find $s=0=\mathfrak{b}$ and $v=B_0$. Hence using the first equation in \eqref{B1-case-1}, \eqref{Bn-final-case-1} follows. As consequence, taking $n=1$ in \eqref{Bn-final-case-1} and comparing the result with the expression for $B_1$ given by \eqref{B1-case-1}, we obtain 
\begin{align}\label{c1-case-1}
C_1=(q^{1/2}-1)(r+c_1)c_1.
\end{align}
Since $C_n =c_n/t_n$, from \eqref{power-pi1} and \eqref{Bn-final-case-1}, we find 
\begin{align}\label{cn-first-case-1}
C_{n+1}=\frac{C_1}{(q-1)c_1}\left(1-q^{n+1} \right)\left(r-\frac{r+c_1}{1+q^{1/2}}(1+q^{(2n+1)/2)}  \right).
\end{align}  
Taking into account that $\mathfrak{a}=-1/(2u)$ and $\mathfrak{b}=0$, using \cite[(4.5)]{KCDMJPArxiv2021a} (for $\mathfrak{c}_3=0$ and $\mathfrak{c}_1=\mathfrak{c}_2=1/2$), we  also have
\begin{align}\label{cn-second-case-pi-1}
C_{n+1}=\left( 1-q^{n+1}\right)\left(\frac{1}{4}(1-q^n) +\frac{C_1}{1-q} q^n \right),\quad n=0,1,2,\ldots.
\end{align}
If $c_1 =rq^{1/2}$ then \eqref{cn-first-case-1} becomes $C_{n+1}=C_1(1-q^{n+1})q^n/(1-q)$ which is incompatible with \eqref{cn-second-case-pi-1}. Thus $c_1 \neq rq^{1/2}$. Comparing the expressions for $C_{n+1}$ given in \eqref{cn-first-case-1} and \eqref{cn-second-case-pi-1} yields
\begin{align}\label{c11-case-1}
C_1= (1-q)\frac{(1+q^{1/2})c_1}{4(c_1-q^{1/2}r)}.
\end{align}
Therefore, combining \eqref{c11-case-1} with \eqref{c1-case-1}, we see that $r+c_1$ is a solution of the following quadratic equation
\begin{align}
2Y^2 -2(1+q^{-1/2})c_1Y -1-\alpha =0.\label{quadratic-eq-case-1}
\end{align}
Let $c$ and $d$ be two complex numbers defined by
\begin{align*}
(c,d)~ \textit{or}~(d,c) \in \left\lbrace \left(c_1-\sqrt{\Delta} ,q^{-1/2}( c_1-\sqrt{\Delta})\right),~\left(c_1+\sqrt{\Delta}, q^{-1/2}(c_1+\sqrt{\Delta})\right)  \right\rbrace,
\end{align*}
where 
$\Delta=c_1 ^2 +q^{1/2}$. Note that $cd\neq 0$. Set $Y_1=(c+d)/2$  and $Y_2=-(c^{-1}+d^{-1})/2$. Hence $Y_1$ and $Y_2$ are solutions of \eqref{quadratic-eq-case-1}. Without loss of generality we may set $r+c_1=Y_1$ and so $Y_1+Y_2=(1+q^{-1/2})c_1$, which yields 
$$r=(c+d)\frac{1+cdq^{-1/2}}{2cd(1+q^{-1/2})},\quad c_1= (c+d)\frac{cd-1}{2cd(1+q^{-1/2})}.$$
Hence \eqref{c11-case-1} (or \eqref{c1-case-1}) becomes $C_1=(1-q)(1-cd)/4$. Consequently, from \eqref{cn-first-case-1} (or \eqref{cn-second-case-pi-1}) and \eqref{Bn-final-case-1}, we obtain
\begin{align}\label{bn-cn-final-case-1}
B_n =(c+d)q^n/2,\quad C_{n+1}=(1-q^{n+1})(1-cdq^n)/4,
\end{align}
together with $k_2=q^{1/2}c_1/C_1=2u(c+d)/(cd(1+q^{-1/2})).$ Using \eqref{bn-cn-final-case-1}, equation \eqref{eq4S} now reads as
\begin{align}\label{Cn-case-1}
(q^{-1}+q^{-1/2})(C_{n+1}-1/4) -2(1+\alpha)(C_n -1/4)& +(q+q^{1/2})(C_{n-1} -1/4)\nonumber\\
& =(\alpha-1)(\alpha+1/2)(c+d)^2 q^{2n-1}.
\end{align}
From $c^2+d^2=2\alpha cd$, it is easy to see that $B_n$ and $C_{n+1}$, in \eqref{bn-cn-final-case-1}, satisfy \eqref{Cn-case-1}. \eqref{eq5S} in this case ($a_n=0$ for $n=0,1,2,\ldots$) reads as
\begin{align}\label{Case-pi-1-Last-eq}
2(1-\alpha)b_nB_n ^2 +(1-2\alpha)(c_n +c_{n+1})B_n  +c_{n+1}B_{n+1}+c_nB_{n-1}\nonumber \\
+(b_n-b_{n+2})(C_{n+1}-1/4)+(b_n-b_{n-2})(C_{n}-1/4) =0,
\end{align}
where $c_n=t_nC_n=k_2q^{-n/2}C_n$ for $n=1,2,\ldots$. Similarly, one may check that \eqref{Case-pi-1-Last-eq} is also satisfied and, therefore, the system of equations \eqref{eq1S-eq2S}-\eqref{eq5S} is fulfilled.  
By a similar argument, if $k_2=0$, we obtain \eqref{bn-cn-final-case-1} with $q$ replaced by $1/q$ and $c^2+d^2-2\alpha cd=0$ as solution of the system of difference equations \eqref{eq1S-eq2S}-\eqref{eq5S}. Thus
\begin{align}\label{final-answer-pi-1}
P_n (X)=Q_n \left(X;c,d|q\right)\quad \text{or}\quad P_n(X)= Q_n\left( X;c,d|q^{-1}\right),\quad n=0,1,2,\dots,
\end{align}
with $c^2+d^2-2\alpha cd=0$, i.e. $c/d=q^{\pm 1/2}$, and this is precisely the assertion of Theorem \ref{T1}.
\end{proof}

\begin{proof}[Proof of Theorem \ref{T2}]
Taking successively $n=1$ and $n=2$ in \eqref{case-pi-2} using \eqref{TTRR} and \cite[(2.40)]{KCDMJPArxiv2021a}  (for $\mathfrak{c}_3=0$ and $\mathfrak{c}_1=\mathfrak{c}_2=1/2$) we obtain the following:
\begin{align}
&B_0=b_1+r+s,\quad c_1=(B_0-r)(B_0-s), \label{b1-pi-2-case-2}\\
&b_2=(2\alpha -1)(B_0+B_1)-2\alpha (r+s),\label{b2-pi-2-case-2}\\
&rs(B_0+B_1 )=c_2B_0 -b_2(B_0B_1-C_1),\label{c2-pi-2-case-2}\\
&c_2=b_2(B_0+B_1)-2\alpha (B_0B_1-C_1)+(r+s)(B_0+B_1)+2\alpha rs. \label{c2-b2-pi-2-case-2}
\end{align}
Solving \eqref{eq3S} we find
\begin{align}\label{first-Bn-pi-2-a}
B_n =\frac{r_0r_1B_0q^{n/2} +\widehat{K}_b\Big(\widehat{a} q^n +(\widehat{b}-\widehat{a} q^{-1/2})q^{n/2}-\widehat{b} q^{-1/2}\Big)}{(\widehat{a} q^{n} +\widehat{b})(\widehat{a} q^{n+1} +\widehat{b} )}q^{(n+1)/2}\; ,
\end{align}
for $n=0,1,2,\ldots$, where $\widehat{K}_b=(r_2B_1 -r_0B_0)/(1-q^{-1/2})$. Since $\widehat{a}~\widehat{b}\neq 0$ and $0<q<1$, then $\lim_{n \rightarrow \infty }q^{-n/2}B_n=-\widehat{K}_b/\widehat{b}$. Evaluating this limit using \cite[(4.4)]{KCDMJPArxiv2021a} (for $\mathfrak{c}_3=0$ and $\mathfrak{c}_1=\mathfrak{c}_2=1/2$), we find $\widehat{K}_b=0$, because $1-2\mathfrak{a}u\neq 0$. Hence \eqref{first-Bn-pi-2-a} reduces to 
\begin{align}\label{expicit-Bn-pi-2}
B_n=\frac{r_0r_1B_0}{r_nr_{n+1}}, \quad n=0,1,2,\dots.
\end{align}
It is immediate that 
$S_n =\sum_{j=0} ^{n-1}B_j=r_1B_0a_n/r_n$ for $n=0,1,2,\ldots$. Comparing this with the result obtained using \cite[(4.4)]{KCDMJPArxiv2021a} (for $\mathfrak{c}_3=0$ and $\mathfrak{c}_1=\mathfrak{c}_2=1/2$) we have 
\begin{align}
(2\mathfrak{a}ur_1q^{-1/2}+\widehat{b}q^{-1})B_0=-2\widehat{a} u\phi'(0),\quad (2\mathfrak{a}ur_1q^{1/2}-\widehat{a}q)B_0=-2 \widehat{b}u\phi'(0).\label{phi-psi-r1-1a-1b}
\end{align}
{\sc Case 1. } Suppose $\phi'(0)=0$. From \eqref{phi-psi-r1-1a-1b} we obtain $r_1B_0=0$. But from \eqref{def-R1} we obtain $0\neq c_1+a_1C_1=r_1C_1$ , i.e. $r_1\neq 0$, and so $B_0=0$. (Conversely, if we assume $B_0=0$, by  (\ref{cond-conject2-final}), we obtain $\phi'(0)=0$.) Hence $B_0=0$ and $\mathfrak{b}=0$. From this, we use \cite[(4.4)-(4.5)]{KCDMJPArxiv2021a} (for $\mathfrak{c}_3=0$ and $\mathfrak{c}_1=\mathfrak{c}_2=1/2$) to obtain $B_{n}=0$ and
\begin{align}
C_{n+1}=\frac{(1-q^{n+1})(1-h q^{n-1})(1+wq^n+hq^{2n})}{4(1-hq^{2n-1}) (1-hq^{2n+1})},\label{cn+1-false-pi-2}
\end{align}
with $h=-(1+2\mathfrak{a}u)/(1-2\mathfrak{a}u)$ and $w=4u(2(\mathfrak{a}+\alpha)C_1-\mathfrak{a})/(2\mathfrak{a}u-1)$. Now defining $\widehat{C}_n=C_n -1/4$, \eqref{eq4S} reads as 
\begin{align}\label{Cn-pi-2-false-1}
(r_{n+1}+r_{n+2})\widehat{C}_{n+1} -2(1+\alpha)r_n\widehat{C}_{n} +(r_{n-1}+r_{n-2})\widehat{C}_{n-1} =0.
\end{align} 
Therefore, we may write
\begin{align}\label{Cn1bb-pi-2}
C_{n+1}= \frac{1}{4} +  \frac{\widehat{\theta}_0\widehat{\theta}_1\widehat{C}_1 q^{n/2} +\widehat{K}_c\Big(\widehat{r}_0 q^n +(\widehat{r_1}-\widehat{r}_0q^{-1/2})q^{n/2}-\widehat{r}_1q^{-1/2}\Big)}{(\widehat{r}_0q^{n} +\widehat{r}_1)(\widehat{r}_0q^{n+1} +\widehat{r}_1)q^{-(n+1)/2}} ,
\end{align}
for $n=2,3,\ldots$, for some complex numbers $\widehat{\theta}_0$, $\widehat{\theta}_1$ and $\widehat{K}_c$, where $\widehat{\theta}_n=r_n+r_{n+1}=\widehat{r}_0q^{n/2}+\widehat{r}_1q^{-n/2}$.
Taking into account \eqref{cond-conject2-final} one may see that \eqref{cn+1-false-pi-2} and \eqref{Cn1bb-pi-2} are compatible if and only if either
\begin{align}\label{data-for-Chebyschev}
C_1=1/4,\quad \widehat{K}_c=0,\quad \mathfrak{a}=\alpha(4\alpha^2-3)/(4\alpha^2-1),
\end{align} 
or
\begin{align}\label{data-for-Chebyschev-b}
C_1=1/2,\quad \mathfrak{a}=\alpha,\quad \widehat{K}_c=0,\quad  \widehat{\theta}_0\widehat{\theta}_1=0.
\end{align}
In case of \eqref{data-for-Chebyschev} holds, we get 
$$
B_n=0,\quad C_{n+1}=1/4,\quad n=0,1,\dots.
$$
This satisfies \eqref{eq1S-eq2S}-\eqref{eq5S}, taking into account that from \eqref{power-pi2} we have $a_n=\gamma_n$ and $b_n=-(r+s)\gamma_n$ for $n=0,1,\ldots$. From \eqref{b1-pi-2-case-2}-\eqref{c2-b2-pi-2-case-2}, we obtain $r+s=0$, $c_1=-r^2$ and $c_2=-\alpha(2r^2-1/2)$. In addition, $r^2=\alpha^2$. In fact, this follows by comparing the expression of $\mathfrak{a}$ given in \eqref{data-for-Chebyschev} and the one obtained from \eqref{coef-a-b-conject-classical}. So we now have $c_1=-\alpha^2$ and $c_2=-\alpha(2\alpha^2-1/2)$. Next, with these expressions, $k_1$ and $k_2$ given in \eqref{tn-case-pi-2} become $k_1=-u(1+q)$ and $k_2=u(1+q^{-1})$ and consequently $c_n=-\alpha \gamma_{n+1}/2$ for $n=0,1,\ldots$.
However taking $n=3$ in \eqref{case-pi-2} and using the fact that $B_n=0$ and $r=-s$, we obtain
\begin{align}
2(2\alpha^2 -1)(C_1+C_2)=\alpha^2-1 +c_3+(4\alpha^2-1)r^2,~c_3C_1=(1-\alpha^2 -C_1-C_2)r^2 .\label{take-n=3-case-pi-2}
\end{align}
It is then clear that the obtained values of $C_1,\,C_2,\,c_3$ and $r^2$ do not satisfy \eqref{take-n=3-case-pi-2}.

For the case where conditions \eqref{data-for-Chebyschev-b} meet, we obtain
$$B_{n-1}=0,~~C_1=1/2,~~C_{n+1}=1/4,\quad\quad~~~n=1,2,\ldots\;.$$
Again this satisfies \eqref{eq1S-eq2S}-\eqref{eq5S} taking into account that from \eqref{power-pi2} we have $a_n=\gamma_n$ and $b_n=-(r+s)\gamma_n$ for $n=0,1,\ldots$. From \eqref{b1-pi-2-case-2}-\eqref{c2-b2-pi-2-case-2}, we obtain $r+s=0$,  $c_1=-r^2$ and $c_2=\alpha(1-2r^2)$. In addition, from the expression of $\mathfrak{a}$ giving in \eqref{data-for-Chebyschev-b}, we write
$$
\alpha=\mathfrak{a}=\frac{(c_2+2\alpha C_2)C_1}{(c_1+C_1)C_2}-\alpha =\alpha\frac{5-6r^2}{1-2r^2}.
$$
Therefore $r^2=1$, and we have $c_1=-1$ and $c_2=-\alpha$. Hence $k_1$ and $k_2$ given in \eqref{tn-case-pi-2} become $k_1=-k_2=-2u$ and, consequently, $t_n=-2\gamma_n$ so that $c_1=-1$ and $c_n=-\gamma_{n}/2$ for $n=1,2,\ldots$.
We check at once that $C_1,~C_2,~c_3$ and $r^2$ satisfy \eqref{take-n=3-case-pi-2}. Hence 
$P_n (X)=\widehat{ T}_n \left(X\right)$ for $n=0,1,2,\ldots$, where $(\widehat{T}_n)_{n \geq 0}$ is the monic Chebyschev polynomials of the first kind, and this is precisely the first assertion of Theorem \ref{T2}.

{\sc Case 2. } Suppose $\phi'(0)\neq 0$. \eqref{phi-psi-r1-1a-1b}, taking into account \eqref{cond-conject2-final}, implies $B_0 \neq 0$. The converse is also true. Hence
$$r_1B_0 (\widehat{a}q^{1/2}-\widehat{b}q^{-1/2} )\neq 0.$$
 Solving \eqref{phi-psi-r1-1a-1b}, we get
\begin{align}\label{a-B_0-from-phi-pi-2}
\mathfrak{a}=-\frac{1+(q\widehat{a}/\widehat{b})^2}{2u\Big(1-(q\widehat{a}/\widehat{b})^2 \Big)},\quad B_0=2u\phi'(0)\frac{1-(q\widehat{a}/\widehat{b})}{1+(q\widehat{a}/\widehat{b})}.
\end{align}
Considering $\widehat{a}$, $\widehat{b}$, $B_0$ and $C_1$ as free parameters, let us define, without loss of generality, two complex numbers $a$ and $b$  such that $-q^{a/2}$ and $q^{b/2}$ are solutions of the following quadratic equation
\begin{align}
Y^2 +\frac{2r_1B_0q^{1/4}}{\widehat{b}(1+q^{1/2})}~Y +~\frac{\widehat{a}}{\widehat{b}}~=0.\label{Equation_for_def_a_b}
\end{align}
Thus
\begin{align}
q^{a/2}-q^{b/2}&=\frac{2r_1B_0q^{1/4}}{\widehat{b}(1+q^{1/2})},\quad\quad~q^{(a+b)/2}=-\widehat{a}/\widehat{b}.\label{def_qab}
\end{align}
On the other hand, we have
\begin{align}\label{defi-ab-pi-2}
(q^{a/2},~q^{b/2}) \in \left\lbrace \left( \frac{\widehat{a}(1+q^{1/2})}{r_1B_0q^{1/4}+\sqrt{\Delta}} ,
~-\frac{r_1B_0q^{1/4}  +\sqrt{\Delta}}{\widehat{b}(1+q^{1/2})}\right),\right. \\
 \left. \left( \frac{\widehat{a}(1+q^{1/2})}{r_1B_0q^{1/4}-\sqrt{\Delta}} ,~ -\frac{r_1B_0q^{1/4}-\sqrt{\Delta}}{\widehat{b}(1+q^{1/2})}\right)     \right\rbrace ,\nonumber
\end{align}
where $\Delta =q^{1/2}B_0^2r_1 ^2 -\widehat{a}\widehat{b}(1+q^{1/2})^2 $. From \eqref{def_qab}, \eqref{a-B_0-from-phi-pi-2} and \eqref{value-d-e-a}, we obtain  
\begin{align*}
\mathfrak{a}&=-\frac{1+q^{a+b+2}}{2u(1-q^{a+b+2})},~B_0 =\frac{(1+q^{1/2})q^{1/4}(q^{a/2}-q^{b/2})}{2(1-q^{(a+b+2)/2})},\\
\mathfrak{b}&=\frac{q^{3/4}(q^{a/2}-q^{b/2})q^{(a+b+2)/2}}{2u^2(q^{1/2}-1)(1-q^{a+b+2})(1-q^{(a+b+2)/2})}.
\end{align*}
 Note that \eqref{rn-case-pi-2} can be written using \eqref{def_qab} as
\begin{align}\label{rn-b-pi-2}
r_n=\widehat{b}(1-q^{n+(a+b)/2})q^{-n/2}, \quad n=0,1,2,\dots.
\end{align}
Therefore \eqref{expicit-Bn-pi-2} becomes
\begin{align}\label{final-and-exact-Bn-pi-2}
B_n&=q^{1/4}(1+q^{1/2})\frac{(1-q^{(a+b)/2})(q^{a/2}-q^{b/2})q^n}{2(1-q^{(2n+a+b)/2})(1-q^{(2n+a+b+2)/2})}\\
&= \Big(q^{(2a+1)/4}+q^{-(2a+1)/4} -y_n-z_n \Big)/2.\nonumber 
\end{align}
Taking into account the above, \eqref{value-d-e-a} becomes 
\begin{align*}
\phi(X)&= -\frac{1+q^{a+b+2}}{2u(1-q^{a+b+2})}X^2 +q^{1/4}(1+q^{1/2})\frac{(q^{a/2}-q^{b/2})}{4u(1+q^{(a+b+2)/2})} X \\
&\quad+ \frac{(1+\alpha)q^{(a+b+4)/2}(q^{a/2}-q^{b/2})^2}{2u(1-q^{a+b+2})(1-q^{(a+b+2)/2})^2} -q^{-1/2}\frac{1-q^{a+b+3}}{1-q^{a+b+2}}C_1  ,\\
\psi(X)&=X-q^{1/4}(1+q^{1/2})\frac{(q^{a/2}-q^{b/2})}{2(1-q^{(a+b+2)/2})}.
\end{align*}
Let ${\bf u}$ be the regular linear functional with respect to the monic OPS $(P_n)_{n\geq 0}$. By \cite[(4.2)]{KCDMJPArxiv2021a} (for $\mathfrak{c}_3=0$ and $\mathfrak{c}_1=\mathfrak{c}_2=1/2$), the regularity conditions for ${\bf u}$ are
$$
(1-q^{n+a+1})(1-q^{n+b+1})(1-q^{n+a+b+1})C_1 \neq 0,  \quad n=0,1,2,\dots.
$$  
Moreover, by \cite[(4.4)-(4.5)]{KCDMJPArxiv2021a} (for $\mathfrak{c}_3=0$ and $\mathfrak{c}_1=\mathfrak{c}_2=1/2$), we obtain the same expression for $B_n$ and
\begin{align*}
C_{n+1}&= \frac{K(1-q^{n+1})(1-q^{n+a+b+1})(1-q^{n+a+1})(1-q^{n+b+1})}{(1-q^{(2n+a+b+1)/2})(1-q^{(2n+a+b+2)/2})^2(1-q^{(2n+a+b+3)/2})},
\end{align*}
where $$K=-\frac{q^{-1/2}uC_1(1-q^{(a+b+3)/2})(1-q^{(a+b+2)/2})^2}{(1-q^{a+1})(1-q^{b+1})(1+q^{(a+b+1)/2})}.$$ 
 After tedious calculations, we can see that this expression for $C_{n+1}$ and \eqref{final-and-exact-Bn-pi-2} satisfy \eqref{eq4S} if and only if 
\begin{align*}
C_1=\frac{(1-q)(1-q^{a+1})(1-q^{b+1})(1+q^{(a+b+1)/2})}{4(1-q^{(a+b+3)/2})(1-q^{(a+b+2)/2})^2}.
\end{align*} 
Alternatively,  $C_1$ follows by taking $n=2$ in \eqref{eq4S}, using $C_2$ and $C_3$. Note that \eqref{eq4S} holds for each $n=2,3,\ldots$. Consequently, we obtain
\begin{align}\label{final-exact-Cn-pi-2}
C_{n+1}&= \frac{(1-q^{n+1})(1-q^{n+a+b+1})(1-q^{n+a+1})(1-q^{n+b+1})}{4(1-q^{(2n+a+b+1)/2})(1-q^{(2n+a+b+2)/2})^2(1-q^{(2n+a+b+3)/2})}=y_nz_{n+1}/4 \; .
\end{align}
Taking into account that $c_n=t_nC_n$, \eqref{b1-pi-2-case-2}-\eqref{c2-b2-pi-2-case-2} yields 
\begin{align*}
&\widehat{b}q^{-1/2}(1-q^{(a+b+2)/2})C_1 +(r+s)B_0 -rs=B_0 ^2 +C_1,\\
&\widehat{b}q^{-1}(1-q^{(a+b+4)/2})B_0C_2 -rs(B_0+B_1)+2\alpha(B_0B_1-C_1)(r+s)\\ 
&\quad \quad\quad\quad\quad\quad\quad\quad\quad\quad\quad\quad\quad\quad\quad= (2\alpha-1)\Big((B_0B_1-C_1)(B_0+B_1) +B_0C_2\Big),\\
&\widehat{b}q^{-1}(1-q^{(a+b+4)/2})C_2 -2\alpha rs +(2\alpha-1)(r+s-B_0-B_1)(B_0+B_1) \\
&\quad\quad\quad\quad\quad \quad\quad\quad\quad\quad\quad\quad\quad\quad\quad=2\alpha (C_1-B_0B_1)+(2\alpha-1)C_2.
\end{align*}
Solving firstly the above system for $\widehat{b}$, $r+s$ and $rs$, and secondly the obtained result for $r$ and $s$, we get $ \widehat{b}=uq^{1/2}(1+q^{-(a+b+2)/2})$ and
\begin{align*}
(r,s)\quad \text{or}\quad (s,r) \in \left\lbrace \left((q^{(2a+1)/4}+q^{-(2a+1)/4})/2,\,-(q^{(2b+1)/4}+q^{-(2b+1)/4}) \right)/2    \right\rbrace\; .
\end{align*}
So \eqref{rn-b-pi-2} becomes $r_n =u(1+q^{-(a+b+2)/2})(1-q^{n+(a+b)/2})q^{(1-n)/2}$ and $t_n =u(1+q^{-(a+b+1)/2})(1-q^{n+(a+b+1)/2})q^{-n/2}$. Using \eqref{final-and-exact-Bn-pi-2}, \eqref{bn-case-pi-2} becomes 
\begin{align}
b_n=(q^{a/2}-q^{b/2})\gamma_n \frac{(1+q^{n+a+b+1/2})}{2(1-q^{n+(a+b)/2})}q^{-(2a+2b+1)/4} ,\label{final-exact-bn-pi-2}
\end{align}
for each $n=0,1,2,\ldots$. Also, since $c_n=t_nC_n$, we obtain
\begin{align}
c_n=-\gamma_n \frac{(1-q^{n+a})(1-q^{n+b})(1-q^{n+a+b})(1+q^{-(a+b+1)/2})}{4(1-q^{n+(a+b-1)/2})(1-q^{n+(a+b)/2})^2},\label{final-exact-cn-pi-2}
\end{align}
for $n=0,1,2,\ldots$.
Finally, again after tedious computations, from \eqref{final-and-exact-Bn-pi-2}, \eqref{final-exact-Cn-pi-2}, \eqref{final-exact-bn-pi-2} and \eqref{final-exact-cn-pi-2}, we see that \eqref{eq5S} holds. Hence \eqref{eq1S-eq2S}-\eqref{eq5S} also hold. Note that \eqref{a-B_0-from-phi-pi-2} may also write as
\begin{align*}
\mathfrak{a}=\frac{1+(q^{-1}\widehat{b}/\widehat{a})^2}{2u\Big(1-(q^{-1}\widehat{b}/\widehat{a})^2 \Big)},\quad B_0= -2u\phi'(0)\frac{1-(q^{-1}\widehat{b}/\widehat{a})}{1+(q^{-1}\widehat{b}/\widehat{a})}.
\end{align*}
Proceeding similarly with the same parameters $a$ and $b$ as defined in \eqref{defi-ab-pi-2}, we obtain the same results with $q$ replaced by $q^{-1}$. Thus
$$P_n(X)=\widehat{P}_n ^{(a,b)}(X|q)\quad \text{or}\quad P_n(X)=\widehat{P}_n ^{(a,b)}(X|q^{-1}),$$
where $\widehat{P}_n ^{(a,b)}(\cdot|q)$ is the monic continuous $q$-Jacobi polynomial, and this is precisely the second assertion of Theorem \ref{T2}.
\end{proof}

\section*{Acknowledgments}
The authors thank to Professor T. H. Koornwinder for helpful discussions and comments. This work is partially supported by the Centre for Mathematics of the University of Coimbra - UIDB/00324/2020, funded by the Portuguese Government through FCT/MCTES.

\end{document}